\documentclass[oneside,11pt]{amsart}
\usepackage{amsmath, amsfonts,amsthm,times,graphics}

\usepackage[active]{srcltx}
 \makeatletter
\renewcommand*\subjclass[2][2000]{%
  \def\@subjclass{#2}%
  \@ifundefined{subjclassname@#1}{%
    \ClassWarning{\@classname}{Unknown edition (#1) of Mathematics
      Subject Classification; using '1991'.}%
  }{%
    \@xp\let\@xp\subjclassname\csname subjclassname@#1\endcsname
  }%
}
 \makeatother
\newtheorem{theorem}{Theorem}[section]
\newtheorem{lemma}[theorem]{Lemma}
\newtheorem{corollary}[theorem]{Corollary}
\newtheorem{proposition}[theorem]{Proposition}
\theoremstyle{definition}

\newtheorem{example}[theorem]{Example}
\newtheorem{remark}[theorem]{Remark}
\numberwithin{equation}{section}
\newcommand{\abs}[1]{\lvert#1\rvert}
 \makeatletter
\renewcommand*\subjclass[2][2000]{%
  \def\@subjclass{#2}%
  \@ifundefined{subjclassname@#1}{%
    \ClassWarning{\@classname}{Unknown edition (#1) of Mathematics
      Subject Classification; using '1991'.}%
  }{%
    \@xp\let\@xp\subjclassname\csname subjclassname@#1\endcsname
  }%
}

\newcommand{\R}{{\mathbb{R}}}

\renewcommand{\Re}{{\rm Re}}       % real part
\newcommand{\eps}{{\varepsilon}}

\def\NABLA#1{{\mathop{\nabla\kern-.5ex\lower1ex\hbox{$#1$}}}}
\def\Nabla#1{\nabla\kern-.5ex{}_{#1}}
\def\Tabla#1{\Tilde\nabla\kern-.5ex{}_{#1}}
\def\abs#1{\mathopen|#1\mathclose|}

\renewcommand{\Tilde}{\widetilde}

 \makeatother

\begin{document}
\title[Radial extension of a bi-Lipschitz parametrization of a curve]{Radial extension of a bi-Lipschitz parametrization of a starlike Jordan curve}
\author{David Kalaj}
\address{Faculty of natural sciences and mathematics, University of Montenegro,
Cetinjski put b.b. 81000, Podgorica, Montenegro}
\email{davidk@t-com.me}

%\author{Djiordjije Vujadinovi\'c}
%\address{University of Montenegro, faculty of natural sciences and mathematics,
%Cetinjski put b.b. 81000, Podgorica, Montenegro}
%\email{djordjijevuj@t-com.me}

%\dedicatory{This paper is dedicated to our authors.}
\begin{abstract}
In this paper we discuss the radial extension $w$ of a bi-Lipschitz
parameterization $F(e^{it})=f(t)$ of a starlike Jordan curve
$\gamma$ w.r.t. the origin. We show that if the
parameterization is bi-Lipschitz, then the extension $w$ is
bi-Lipschitz and consequently quasiconformal for some constant
$\mathcal{K}_w\ge 1$. If $\gamma$ is the unit circle, then
$\mathrm{Lip}(f)=\mathrm{Lip}(F)=\mathrm{Lip}(w)$ and
$\mathcal{K}_w=\max\{\mathrm{Lip}(F),\mathrm{Lip}(F^{-1})\}$. Finally it is studied quasiconformal and Lipschitz behavior of a radial extension of an arbitrary parametrization of a starlike Jordan curve.
\end{abstract}
\maketitle %\tableofcontents
\section{Introduction}
Conformal mappings between plane domains are mappings which preserve
the angles between smooth curves. Quasiconformal mappings,
intuitively speaking are mappings which "quasi-preserve" the
angles. Quasiconformal mappings make one of most powerful connection
between complex analysis,  geometry and PDE. One of the most
important application of quasiconformal mappings is the conformal
representation of Euclidean surfaces. In this paper we will describe
all quasiconformal mappings between the unit disk and a starlike
domain with respect to the origin which are radial extensions of
boundary mappings.
\subsection{Quasiconformal mappings}
By $\mathbf U$ we denote the unit disk in the complex plane $\mathbf
C$. Its boundary is the unit circle $\mathbf T$. Let $D$ and
$\Omega$ be subdomains of the complex plane $\mathbf C$, and
$w=u+iv:D\to \Omega$ be a mapping that has both partial derivatives
at a point $z\in D$. We define the Jacobian matrix by
$$\nabla w(z)=\begin{pmatrix}
  u_{x} & u_{y} \\
  v_{x} & v_{y}
\end{pmatrix}.$$ Its operator norm and the Hilbert-Schmidt norm are
given by: \begin{equation*}|\nabla w(z)|:=\max_{|h|=1}|\nabla
w(z)h|=|w_z|+|w_{\bar z}|,\end{equation*} and $$\|\nabla
w\|:=\sqrt{|w_x|^2+|w_y|^2}.$$ We define the function
\begin{equation*}l(\nabla w)(z):=\min_{|h|=1}|\nabla w(z)h|=||w_z|-|w_{\bar
z}||,\end{equation*} and the Jacobian
\[J_w(z)=|w_z|^2-|w_{\bar z}|^2,\]
where
$$w_z := \frac{1}{2}\left(w_x+\frac{1}{i}w_y\right)\text{ and } w_{\bar z} := \frac{1}{2}\left(w_x-\frac{1}{i}w_y\right).$$
%We say that a function $u:D\to \mathbf R$ is $ACL$ (absolutely
%continuous on lines) in  the region $D$, if for every closed
%rectangle $R\subset D$ with sides parallel to the $x$ and $y$-axes,
%$u$ is absolutely continuous on a.e. horizontal and a.e. vertical
%line segment in $R$. Such a function has of course, partial
%derivatives $u_x$, $u_y$ a.e. in $D$.

An orientation preserving homeomorphism $w:D\to \mathbf{C}$ of
Sobolev class $W^{1,1}_{\mathrm{ loc}} (D ; \mathbf{C})$ is said to
be $K$-quasiconformal, $1\le K < \infty$, if
\begin{equation}\label{defqc} D_w(z):=\frac{|\nabla
w(z)|}{l(\nabla w)(z)}= \frac{|w_z|+|w_{\bar z}|}{|w_z|-|w_{\bar
z}|} \le K \ \ \ \text{a.e. on $D$},\end{equation} (cf. \cite{Ahl},
pp. 23--24). Notice that the condition (\ref{defqc}) can be written
as
$$|\mu_w(z)|\le k\quad \text{a.e. on $D$ where
$k=\frac{K-1}{K+1}$},$$ and $\mu_w(z):={w_{\bar z}}/{w_z}$ is the
complex dilatation of $w$ and in its equivalent form $${\|\nabla
w\|^2}\le \left(K+\frac{1}{K}\right){J_w}.$$  Sometimes instead of
$K$ quasiconformal we write $k$ quasiconformal.

\subsection{Lipschitz continuity}
A mapping $f:X\to Y$, between metric spaces $(X,d_X)$ and $(Y,d_Y)$
is said to be $\mathcal L-${\it Lipschitz} and $(\ell,\mathcal
L)-${\it bi-Lipschitz}, for some constants $0<\ell\le \mathcal L$,
if
$$d_Y(f(x),f(y))\le \mathcal L d_X(x,y), \ \ \textrm{for all}  \quad x,y\in X,$$ and $$\ell d_X(x,y)\le d_Y(f(x),f(y))\le \mathcal L\,d_X(x,y),
\ \ \textrm{for all} \quad x, y\in X$$ respectively. If $\ell=1/L$, then we say that $f$ is $L-$bi-Lipschitz. We define
$$\mathrm{Lip}(f): = \sup_{x\neq
y}\frac{d_Y(f(x),f(y))}{d_X(x,y)}.$$

\subsection{Quasiconformal extension}
A homeomorphism $\tilde f: \mathbf R \to \mathbf R$ is called $M-$
quasisymmetric if for all $x$ and $t>0$
$$\frac{\tilde f(x+t)-\tilde f(x)}{\tilde f(x)-\tilde f(x-t)}\le
M,$$ and $f(\infty)=\infty$. We easily can modify the previous
definition for self - homeomorphisms of the unit circle. It is well
known that every quasisymmetric function has quasiconformal
extension to the half-plane. We want to point out two most important
extensions: \emph{Beurling- Ahlfors extension} \cite{ba}, and the
\emph{barycentric extension} of Douady and Earle \cite{dc}.

Recall that a Jordan domain $\Omega$ and its boundary curve $\gamma$ is
\emph{starlike} w.r.t. the origin if
$\overline{\Omega}=\bigcup_{z\in \gamma}[0,z].$ Let $\Omega$ be a
starlike Jordan domain with respect to the origin.
 Let $\gamma=\partial \Omega$ and let $F:\mathbf T\to \gamma$ be a
homeomorphism. The \emph{radial extension} of a homeomorphism is
defined by $w(re^{it})=rF(e^{it})$ and it defines a homeomorphism of
the unit disk onto $\Omega$. Radial extension does not map smoothly.

Note this important and simple fact, if $\Omega$ is not starlike
w.r.t. $0$, then the radial extension is not a mapping between
$\mathbf U$ and $\Omega$. One of primary aims of this paper is to
describe all homeomorphisms, whose radial extensions are
quasiconformal. We will show that the extension is quasiconformal if
and only if it is bi-Lipschitz. In order to do this we will assume
that the Jordan curve is smooth and starlike with respect to the
origin.  It is well known that every bi-Lipschitz is quasiconformal.
The converse is not true. However, if the mapping is quasiconformal,
then it is H\"older continuous under some conditions on the
boundaries (see \cite{mar1} and \cite{vuorinen}). For connection
between these two concepts (bi-Lipschitz mappings and quasiconformal
mappings) we also refer to the paper \cite{bish}. A counterpart of
quasiconformal extensions of quasisymmetric functions is a Lipschitz
extension of Lipchitz functions. We refer for the latter topic to the recent
paper of Kovalev \cite{leonid}.

We say that a mapping $P:\mathbf T\to \gamma$ is a {\it polar
parametrization}, if $\arg P(e^{it})=t$. Thus
$P(e^{it})=r(t)e^{it}$, for some positive continuous function $r$,
such that $r(0)=r(2\pi)$.

For a given homeomorphism $F:\mathbf T\to \gamma$ define

\begin{itemize}
    \item ${f}:[0,2\pi]\to \gamma$, $\ f(t)=F(e^{it})$
    \item $w: \overline{\mathbf U} \to \overline{\Omega}$ and $w\colon\mathbf{C}\to\mathbf{C}$ with   $\ w(z) =|z|F(z/|z|)$.
\end{itemize} We call $w$ \emph{the radial extension} of $F$.
\emph{It follows from the definition that $w$ is a homeomorphism of $\mathbf{C}$ onto itself.}

Take $z=e^{it},w=e^{is}\in \mathbf T$. {\it The restriction of
spherical and the chordal distance on $\mathbf{T}$} between points
$z$ and $w$ are defined by
$$d_1(z,w)=\min\{|\arg z-\arg w|,2\pi-|\arg z-\arg w|\}=\min\{|t-s|,2\pi-|t-s|\} $$  and  $$ d_2(z,w)=|z-w|=2\left|{\sin\frac{t-s}2}\right|.$$ Notice that
$d_1\ge d_2$. For a given function $F$ define the following four
constants:
\begin{itemize}
    \item $l: = \mathrm{Lip}(f)$,
    \item $L :=\mathrm{Lip}(F)$,
    \item $\Lambda:=\mathrm{Lip}(w)$ and
    \item $\mathcal K=\mathrm{ess}\sup_z D_w(z)$.
\end{itemize}
In this paper we will compare these constants.

We will show that if $\gamma=\mathbf T$, then
$l=L=\Lambda\le\mathcal K$ (Theorem~\ref{cir}), provided that $L<\infty$. The condition
$\gamma=\mathbf T$ is essential, see Example~\ref{she}. However, if
$F$ is a polar parametrization of a starlike Jordan curve w.r.t.
$0$, then we will show the following interesting fact $l=L$
(Theorem~\ref{ara}). In the last section, we will show
that the radial extension of a mapping $F$ of the unit circle onto a smooth  Jordan curve $\gamma$, starlike w.r.t. $0$, and with tangent lines disjoint from the origin, is quasiconformal if and only if it is
bi-Lipschitz (Theorem~\ref{star}). Theorem~\ref{star} can be considered
as an extension of a special case (i.e. of the case $n=2$) of a
result of Martio and Srebro \cite{mar2}, where the authors
considered the quasiconformality of radial stretching map, of the
unit ball $B^n$ onto a domain $\Omega\subset\mathbf{R}^n$ satisfying
$\beta-$cone condition; that is radial extension of a polar
parametrization. Finally we provide two explicit examples.
\section{Preliminary results}
In this section we will derive some auxiliary results. Further we
will consider the case $\gamma=\mathbf T$. Since $|z-w|\le |\arg
z-\arg w|,\ \ \text{ for }\ \ z,w\in \mathbf T$ and $\mathbf
T\subset \overline{\mathbf U}$ it follows that
\begin{equation}\label{lll}l \le L \le\Lambda.\end{equation}  Recall
the following fundamental result of Rademacher:
(\cite[Theorem~6.15]{juha}). Every Lipschitz function in an open set of
$\mathbf R^n$ is differentiable almost everywhere.

%\begin{lemma}\label{dri}
%If $\varphi:\mathbf R\to \mathbf R$ is a $\mathcal{L}$ Lipschitz
%mapping, such that $\varphi(x+a) = \varphi(x) + b$ for some $a$ and
%$b$ and every $x$, then there exist a sequence of $C^\infty$
%$\mathcal{L}$ Lipschitz functions $\varphi_n:\mathbf R\to \mathbf R$
%such that $\varphi_n$ converges uniformly to $\varphi$, and
%$\varphi_n(x+a) = \varphi_n(x)+b$.
%\end{lemma}

%\begin{proof}
%This result is well-known. We refer to \cite[Lemma~2.5]{pisa}.
%\end{proof}

By Rademacher theorem,  and Mean value theorem, we have

 \begin{lemma}\label{ura}

For a
Lipschitz mapping $f$ defined in the interval $[0,2\pi]$ and a
Lipschitz mapping $w$ defined in the complex plane, we have
\begin{equation}\label{flip}\mathrm{Lip}(f) = \mathrm{ess}\sup_t|f'(t)|\end{equation} and

\begin{equation}\label{wlip}\mathrm{Lip}(w)
= \mathrm{ess}\sup_{z\in\mathbf{C}}|\nabla w(z)|.\end{equation} In particular if $z=re^{it}$ and $w(z)=rf(t):\mathbf{C}\to \mathbf{C}$, then \begin{equation}\label{expl0}\mathrm{Lip}(w) = \mathrm{ess}\sup_t
\frac{1}{2}\left(|f(t)-i f'(t)| +|f(t)+i {
f'}(t)|\right).\end{equation}
Further if $w(z)=rf(t):\mathbf{C}\to \mathbf{C}$ is a homeomorphism  and $w^{-1}:\mathbf{C}\to \mathbf{C}$  is Lipschitz, then \begin{equation}\label{expl1}\mathrm{Lip}(w^{-1}) = \frac{1}{2}\mathrm{ess}\sup_t
\left||f(t)-i f'(t)| -|f(t)+i {f'}(t)|\right|^{-1}.\end{equation}

\end{lemma}
\begin{proof} First notice that, since $[0,2\pi]$ and $\mathbf{C}$ are convex, then \eqref{flip} and \eqref{wlip} are well-known results. For the completeness we include their proofs here.
In order to prove \eqref{flip}, we observe that if $f$ is Lipschitz then $$f(x)=f(0)+\int_0^x \psi(t)dt,\ \ \ \text{ for some $\psi\in L^\infty([0,2\pi])$}$$ such that $ f'(t)=\psi(t)$ for a.e. $t\in[0,2\pi]$. Thus  $\mathrm{ess}\sup_t|f'(t)|=\mathrm{ess}\sup_t|\psi(t)|$. Hence $$|f(x)-f(y)|=|\int_x^y \psi(t)dt|\le \mathrm{ess}\sup_t|f'(t)| \cdot|x-y|.$$ This implies that $\mathrm{Lip}(f)\le  \mathrm{ess}\sup_t|f'(t)|$. In order to prove the opposite inequality, let $\varepsilon>0$ such that for some $t\in[0,2\pi]$, $|f'(t)|\ge \mathrm{ess}\sup_t|f'(t)|-\varepsilon$. Then $$\mathrm{Lip}(f)\ge \lim_{s\to t} \frac{|f(s)-f(t)|}{|t-s|}=|f'(t)|\ge \mathrm{ess}\sup_t|f'(t)|-\varepsilon.$$ Since $\varepsilon$ is arbitrary, we get $\mathrm{Lip}(f)\ge  \mathrm{ess}\sup_t|f'(t)|$. This concludes \eqref{flip}.

In order to prove \eqref{wlip}, for fixed $z_1$ and $z_2$ take $\lambda(t)=w(z_1+t(z_2-z_1))$. Assume that $w$ is Lipschitz. Then $$|\lambda(t)-\lambda(s)|=|w(z_1+t(z_2-z_1))-w(z_1+s(z_2-z_1))|\le \mathrm{Lip}(w)|z_1-z_2||s-t|.$$ Thus $\lambda$ is Lipschitz and by using the previous proof, by using the formula $\lambda'(t)=\nabla w(z_1+t(z_2-z_1))(z_2-z_1)$ we have $$|w(z_1)-w(z_2)|=|\lambda(1)-\lambda(0)|\le \mathrm{ess}\sup_t|\lambda'(t)|(1-0)\le \mathrm{ess}\sup_{z\in \mathbf{C}}|\nabla w(z)||z_2-z_1|.$$  Thus $$\mathrm{Lip}(w)
\le  \mathrm{ess}\sup_{z\in \mathbf{C}}|\nabla w(z)|.$$ Let $\varepsilon>0$ and assume that for some $z\in \mathbf{C}$, $$|\nabla w(z)|\ge \mathrm{ess}\sup_{z\in \mathbf{C}}|\nabla w(z)|-\varepsilon. $$ Then there is $h$ with $|h|=1$ such that $|\nabla w(z)|=|\nabla w(z) h|$. Also we have $w(z+k)=w(z)+\nabla w(z) k+o(|k|)$. By choosing $k=th$, $t>0$, we obtain $$\frac{|w(z+th)-w(z)|}{|th|}=\frac{|\nabla w(z) h|}{|h|}+O(t).$$ Thus $$\mathrm{Lip}(w)\ge \lim_{t\to 0}\frac{|w(z+th)-w(z)|}{|th|}=|\nabla w(z)|.$$ This and the fact that $\varepsilon>0$ is arbitrary yields \eqref{wlip}.

If $z=re^{it}$, then
$e^{2it}={z}/{\bar z}.$ Thus $2ie^{2it} t_z={1}/{\bar z}$
and  $2ie^{2it} t_{\bar z}=-{z}/{\bar z^2}.$ Therefore
$t_z = {1}/{(2i z)}$ and  $t_{\bar z} = -{1}/{(2i \bar
z)}.$ Moreover $\, r^2 = z\bar z.$ Hence
$r_z = {\bar z}/{(2r)}\,$ and $ \,r_{\bar z} = {z}/{(2r)}.$

Assume that $w(z) = r f(t).$ Then we have $$w_z = r_z f(t) +
rf'(t)t_z\ \ \text{ and }\ \  w_{\bar z} = r_{\bar z} f(t) +
rf'(t)t_{\bar z}.$$ Whence $$ w_z = \frac{\bar z}{2r} f(t) +
f'(t)\frac{r}{2i z}\ \ \text{ and }\ \ \ w_{\bar z} =
\frac{z}{2r}f(t) - f'(t) \frac{r}{2i \bar z}.$$ Thus
\begin{equation}\label{ww}|w_z|=\frac{1}{2}|f(t)-i f'(t)| \text{ and } |w_{\bar
z}|=\frac{1}{2}|f(t)+i f'(t)|\end{equation} and
\begin{equation}\label{nab}|\nabla w(z)|=\frac{1}{2}\left(|f(t)-i
f'(t)| +|f(t)+i f'(t)|\right),\end{equation} which
implies that
\begin{equation}\label{expl}\mathrm{Lip}(w) = \mathrm{ess}\sup_t
\frac{1}{2}\left(|f(t)-i f'(t)| +|f(t)+i {f'}(t)|\right).\end{equation} The proof of \eqref{expl1} is similar but this case we make use of the formula $\nabla w^{-1}(w(z))={\nabla w(z)}^{-1}$ and \eqref{ww} in order to obtain \[\begin{split}|\nabla w^{-1}(w(z))|&=\sup_{h\neq 0}\frac{|{\nabla w(z)}^{-1}h|}{|h|}\\&=\left(\inf_{k\neq 0}\frac{|\nabla w(z)k|}{|k|}\right)^{-1}\\&=\left(\inf_{|k|=1}|\nabla w(z)k|\right)^{-1}\\&=\frac{1}{l(\nabla w(z))}\\&={||w_z|-|w_{\bar z}||}^{-1}\\&=
\frac{1}{2}
\left||f(t)-i f'(t)| -|f(t)+i {
f'}(t)|\right|^{-1}.\end{split}\] We can now apply the previous proof in order to obtain $$\mathrm{Lip}(w^{-1})=\mathrm{ess}\sup_{\omega\in \mathbf{C}} |\nabla w^{-1}(\omega)|.$$ This finishes the proof of lemma.
\end{proof}
 Now we have the following theorem:
%\section{Radial extension of circle homeomorphisms}
\begin{theorem}\label{cir}
If $F(e^{ix})=e^{i\psi(x)}:\mathbf T\to \mathbf T$ is a bi-Lipschitz
mapping and $w$ its radial extension, then
\begin{equation}\label{llk}l=L=\Lambda=\mathrm{Lip}(\psi)=|{\psi'}|_\infty,\end{equation}  \begin{equation}\label{lin}
\mathrm{Lip}(w^{-1})=\mathrm{Lip}(\psi^{-1})=|1/{\psi'}|_\infty
\end{equation} and
\begin{equation}\label{kes}\mathcal
K=\max\{|{\psi'}|_\infty,|1/{\psi'}|_\infty\}.\end{equation} Here and in
the sequel by $|g|_\infty$ we mean the $L^\infty$ norm of $g$.
\end{theorem}
\begin{proof} Prove first that $w$ is Lipschitz if and only if $F$ is Lipschitz. Since one  direction is trivial, it remains to prove that $F$ is Lipschitz implies that $w$ is Lipschitz. For $z_1=r_1e^{it_1}$, $z_2=r_2 e^{it_2}$, $z_1\neq z_2$
\[\begin{split}|w(z_1)-w(z_2)|&=|r_1 F(e^{it_1})-r_2 F(e^{it_2})|\\&\le |r_2-r_1||F(e^{it_1})|+r_2|F(e^{it_1})- F(e^{it_2})|\\&\le |F|_\infty|r_2-r_1|+r_2\mathrm{Lip}(F)|e^{it_1}- e^{it_2}|\\&\le |F|_\infty|z_1-z_2|+\mathrm{Lip}(F)\left|z_2-z_1 \frac{|z_2|}{|z_1|}\right|
\\&\le |F|_\infty|z_2-z_1|+2 \mathrm{Lip}(F)|z_2-z_1|.\end{split}\] In the last inequality we used the following sequence of inequalities
$$\left|z_2-z_1 \frac{|z_2|}{|z_1|}\right|\le |z_2-z_1|+\left|z_1-z_1 \frac{|z_2|}{|z_1|}\right|=|z_2-z_1|+||z_1|-|z_2||\le 2|z_1-z_2|.$$
If $F$ is a mapping of the unit circle onto itself, then $$f(t)=F(e^{it})=
e^{i\psi(t)}$$ for some increasing bijective function
$\psi:[0,2\pi]\to [0,2\pi]$. Moreover, \begin{equation}\label{ft}f(t)+i f'(t)=
e^{i\psi(t)}(1-{\psi'}) \ \ \text{ and }\ \ f(t)-i f'(t)=
e^{i\psi(t)}(1+{\psi'}).\end{equation} Thus by \eqref{nab} we have $$|\nabla w(z)| =
\max\{1,{\psi'(t)}\},$$ and
\begin{equation}\label{opsi}\mathrm{Lip}(w)=|{\psi'}|_\infty.\end{equation} By \eqref{opsi} and \eqref{lll} we obtain \eqref{llk}. To finish the proof we have to notice that $F^{-1}(t)=e^{i\phi(t)}$, and $w^{-1}(z)=r e^{i\phi(t)}$, where $\phi=\psi^{-1}$. Thus $F^{-1}$ is Lipschitz if and only if $w^{-1}$ is Lipschitz. Moreover $(\psi^{-1})'(\psi(s))=1/\psi'(s)$ and this concludes the proof of \eqref{lin}. Thus $w$ is quasiconformal. In order to find the quasiconformality constant we make use of \eqref{ww} and \eqref{ft} in order to obtain $$K=\mathrm{ess}\sup_t\frac{|1+{\psi'}|+|1-{\psi'}|}{|1+{\psi'}|-|1-{\psi'}|}=\max\{|{\psi'}|_\infty,|1/{\psi'}|_\infty\}.$$
\end{proof}
From the previous theorem we infer the following corollaries:
\begin{corollary}\label{qo}\cite{trans}
Let $F(e^{it})=e^{i\psi(t)}: \mathbf T\to  \mathbf T$. Then $F$ is $L-$Lipschitz
continuous if and only if $\psi$ is
$L-$Lipschitz continuous.
\end{corollary}
\begin{corollary}
Let $w(z)=|z| F(z/|z|):\overline{\mathbf{U}}\to \overline{\mathbf{U}}$, where $F: \mathbf
T\to \mathbf T $. Then $w$ is $L-$Lipschitz if and only if $F$ is
$L-$Lipschitz.
\end{corollary}

\begin{remark}
The question arises, can we replace the unit circle by some other
starlike Jordan curve $\gamma$ in the previous statements. The
following example shows that in general we do not have that $l=L$.
\end{remark}
\begin{example}\label{she} Let $F(e^{it})=(-1+\min\{2\pi  - t,
t\}, (1/10)  \sin t).$ And $f(t) = F(e^{it})$. Then
$\mathrm{Lip}(f) =
{\sqrt{101}}/{10}<{\pi}/{2}\le \mathrm{Lip}(F).$ To obtain the Lipschitz constant of $f$ we use \eqref{flip}. We have
$$f'(t)=\left\{
          \begin{array}{ll}
            (1,1/10 \cos t), & \hbox{ if } t>\pi  \\
            (-1,1/10 \cos t), & \hbox{ if } t<\pi.
          \end{array}
        \right.
$$
Then $\mathrm{ess}\sup|f'(t)|={\sqrt{101}}/{10}$. Further for $0\le t\le \pi$ $$\frac{|F(e^{it})-F(-1)|^2}{|e^{it}+1|^2}=\frac{100 (\pi -t)^2+\sin^2t}{200 (1+\cos t)}.$$ Thus $$\frac{|F(1)-F(-1)|^2}{|1-(-1)|^2}=\frac{\pi^2}{4}.$$ This implies the claim of the example.
\end{example}

\section{Polar parametrization}
Let $\gamma$ be a starlike Jordan curve w.r.t.  the origin. Let
$f(t)=r(t)e^{it}$ be the { polar parametrization} of $\gamma$. In
this section we will prove the following interesting results.
\emph{For all polar parametrizations of starlike curves we have $l=L$ (Theorem~\ref{ara}).}

As we noticed before, \begin{equation}l=\sup_{t\neq
s}\frac{|f(s)-f(t)|}{|s-t|}\le L=\sup_{e^{it}\neq
e^{is}}\frac{|f(s)-f(t)|}{|e^{is}-e^{it}|}.\end{equation}

As $|s-t|> |e^{it}-e^{is}|$ for all $e^{it}\neq e^{is}$, one expects
that for some (or all) polar parametrizations $f$ we should have
$l<L$. However we have
\begin{theorem}\label{ara}
Let the function $f(t) = r(t) e^{it}:[0,2\pi]\to\gamma$ defines
a Lipschitz Jordan curve. Then
\begin{equation}\label{limi}\sup_{t\neq
s}\frac{|f(t)-f(s)|}{|e^{it}-e^{is}|}=\sup_s\lim\sup_{t\to s}
\frac{|f(t)-f(s)|}{|e^{it}-e^{is}|}.\end{equation}
\end{theorem}

\begin{remark}
%nontrivial only for Lipschitz continuous $r$.
It is well-known that
a Jordan curve $\gamma$ parameterized by a smooth function $[0,l]\ni
s\to g(s)\in \gamma$ is starlike w.r.t. 0, if and only if
$\mathrm{Im}(\overline{g(s)}{g'}(s))>0$ for every $s$ (see for example
\cite{cie}). In the case of smooth $r$ we have
$\mathrm{Im}(\overline{f(t)}f'(t))=r^2(t)>0$.
\end{remark}

\begin{proof}[Proof of Theorem~\ref{ara}]
Dividing $f$ by a positive constant, if it is necessarily, we may assume that $\mathrm{Lip}(f)= 1$, and therefore, in view of Lemma~\ref{ura} we have
\begin{equation}\label{rrp}r^2(t)+ (r'(t))^2\le 1, \text{for \ \ a.e.} \ \ \ t\in[0,2\pi].\end{equation} Note that $r(t)\le 1$. Our goal is to show that \begin{equation}\label{goal}\frac{|f(t)-f(s)|}{|e^{it}-e^{is}|}\le \mathrm{Lip}(f)\end{equation} i.e.
$$|r(t)e^{i(t-s)}-r(s)|\le |e^{i(t-s)}-1|$$ for all $t$. So we lose no generality in assuming that $s=0$. The previous inequality simplifies to
\begin{equation}\label{triv}2\cos t(1-r(0)r(t))\le 2-r(0)^2-r(t)^2\end{equation} which is trivially true
when $\cos t\le 0$. So we only have to deal with $0<t<\pi/2$.

We lose no generality in assuming that $r(t)\le r(0)$. Otherwise we
can consider the new function $R(s)=r(t-s)$. It is convenient to
write $r(0)=\cos \alpha$ where $0\le\alpha < \pi/2$. Integrating the
following differential inequality
$$r'(t)\ge -\sqrt{1-r(t)^2},  \text{for \ \ a.e.} \ \ \ t\in[0,2\pi] $$ (which follows from \eqref{rrp}), we find that
$$-\int_0^t\frac{dr(t)}{\sqrt{1-r^2(t)}}\le t$$ i.e. $\mathrm{arccos}(r(t))\le t+\alpha$ which can be written as
\begin{equation}r(t)\ge \cos(\alpha+t).\end{equation} Here we use the fact that $ r'\in L^\infty$, which implies the following simple equality $dr(t)=r'(t)dt$. The inequality
\eqref{triv} can be written as
\begin{equation}\label{triv21}r(t)^2-2r(t)\cos \alpha\cos t+\cos^2\alpha-2(1-\cos t)\le
0.\end{equation} So it is enough to show that for $\cos(t+\alpha)\le
R\le \cos\alpha$ the following relation is true
\begin{equation}\label{triv2}R^2-2R\cos \alpha\cos
t+\cos^2\alpha-2(1-\cos t)\le 0.\end{equation} Since the left hand
side is a convex function of $R$, it suffices to verify
\eqref{triv2} at the endpoints $R=\cos(\alpha+t)$ and
$R=\cos\alpha$. Putting $R=\cos(\alpha+t)$ into \eqref{triv2} we
get $-(1-\cos t)^2\le 0$. Putting $R=\cos\alpha$ into \eqref{triv2}
we get $-2\sin^2\alpha (1-\cos t)\le 0$.

To finish the proof observe that \[\begin{split}\mathrm{Lip}(f)&=\mathrm{ess}\sup_{t\neq s}\frac{|f(t)-f(s)|}{|t-s|}\\&=\mathrm{ess}\sup_s|f'(s)|=
\sup_s\lim\sup_{t\to s}
\frac{|f(t)-f(s)|}{|e^{it}-e^{is}|}\\&=\sup_s\lim\sup_{t\to s}
\frac{|f(t)-f(s)|}{|t-s|}.\end{split}\] The last relation follows from the fact that $\lim_{\tau \to 0}\frac{|e^\tau-1|}{|\tau|}=1$. This together with \eqref{goal} implies \eqref{limi}.
\end{proof}

\begin{corollary}\label{pams}
Together with the assumptions of Theorem~\ref{ara} assume that $r$
is a smooth function. Then
$$\sup_{e^{it}\neq
e^{is}}\frac{|f(s)-f(t)|}{|e^{is}-e^{it}|}=\max_{t}\sqrt{r^2(t)+{r'}^2(t)}.$$
\end{corollary}
By Theorem~\ref{ara} we obtain the following corollary:
\begin{corollary}\label{rrr}
Let $\gamma$ be a Jordan starlike curve w.r.t. the origin,
parameterized by polar coordinates $F(e^{it})=r(t)e^{it} :\mathbf
T\to \gamma$. Let
$$w(z) = |z|F\left(\frac{z}{|z|}\right):\overline{\mathbf U}\to \overline{\Omega}, z\neq0,\ \ w(0)=0 $$ be its
radial extension between the unit disk $\mathbf U$ and the Jordan domain
$\Omega=\mathrm{int}(\gamma)$. If $L:=\mathrm{Lip}(F) =
\mathrm{Lip}(w)<\infty$, then $L=\max\{|z|: z\in \gamma\}$.
\end{corollary}
\begin{proof}  Let $f(t)=F(e^{it})$. By using \eqref{flip}, \eqref{expl} and $f'(t)= ie^{it}r(t)+e^{it}r'(t)$ we have $$\mathrm{Lip}(f)=\mathrm{ess}\sup_t|ir(t)+r'(t)|$$ and $$\mathrm{Lip}(w)=\frac{1}{2}\mathrm{ess}\sup_t\left(|r'(t)|+|2r(t) -i
r'(t)|\right).$$ By Theorem~\ref{ara}, $\mathrm{Lip}(f)=\mathrm{Lip}(F)$.
So $\mathrm{Lip}(F) =\mathrm{Lip}(w)$  if and only if  \begin{equation}\label{foles}\mathrm{ess}\sup_t\left(|r'(t)|+|2r(t) -i
r'(t)|\right) =2\mathrm{ess}\sup_t|ir(t)+r'(t)|.\end{equation} From \eqref{foles} it follows that there is a set of zero Lesbegue measure $E$ such that for
$I= [0,2\pi]\setminus E$ we have \begin{equation}\label{fole}\sup_{t\in I}\left(|r'(t)|+|2r(t) -i
r'(t)|\right) =2\sup_{t\in I}|ir(t)+r'(t)|.\end{equation}
%or what is the same  \begin{equation}\label{fol}\mathrm{ess}\sup_t\left(\frac{|r'(t)|+|2r(t) -i
%r'(t)|}{r(t)}\right) =2\mathrm{ess}\sup_t\frac{|ir(t)+r'(t)|}{r(t)}.\end{equation}
Let $$A=2\sup_{t\in I}|ir(t)+r'(t)|$$ and assume that $t_n\in I$ is a sequence of points such that
 $$A=\lim_{n\to \infty}2|ir(t_n)+r'(t_n)|.$$ Then \begin{equation}\label{AA}A=\lim_{n\to\infty} \left(|r'(t_n)|+|2r(t_n) -i
r'(t_n)|\right).\end{equation} In order to prove \eqref{AA}, observe that $|r'(t_n)|+|2r(t_n) -i
r'(t_n)|\ge 2|ir(t_n)+r'(t_n)|.$ Thus $A\le \limsup_{n\to\infty} \left(|r'(t_n)|+|2r(t_n) -i
r'(t_n)|\right)$. But $$\limsup_{n\to\infty} \left(|r'(t_n)|+|2r(t_n) -i
r'(t_n)|\right)\le \sup_{t\in I}\left(|r'(t)|+|2r(t) -i
r'(t)|\right)=A.$$ This implies \eqref{AA}. Since $A<\infty$, there is a subsequence of $t_n$ which will be denoted also by $t_n$ such that $r(t_n)\to R$ and $r'(t_n)\to R'$ such that $$A=|R'|+|2R-I R'|=2|i R+R'|. $$

Since $$\alpha+\sqrt{4+\alpha^2}-2 \sqrt{1+\alpha^2}>0$$ for $\alpha>0$, it follows that $R'=0$. Thus $A= 2R\le 2 \max\{|z|: z\in \gamma\} $. Let $z_n=r(\tau_n)\in F(I)$ such that $\lim_{n\to \infty} |z_n|=\max\{|z|: z\in \gamma\}$. Then $A\ge \limsup_{n\to\infty}2 |r(\tau_n)+ir'(\tau_n)|\ge 2 L$. The conclusion is that $A=2 L=2 \max\{|z|: z\in \gamma\}$. The proof is completed.
\end{proof}

\section{A general parametrization}
%We call a curve $\gamma$ starlike with respect to $0\in\Omega=
%\mathrm{int}(\gamma)$, if $[0, z)\subset \Omega$ for $z\in \gamma$.
Assume that $D$ is a starlike domain in the complex plane $\mathbf{C}$. Let as recall two conditions on $D$.
\begin{subsection}{\bf $\alpha$-tangent condition.}\cite{kvw,gv} Suppose $D$ is a strictly starlike domain w.r.t. the origin $0$ and $x\in\partial D$. For each $z\in\partial D$, $z\neq x$, we let $\alpha(z,x)$
denote the acute angle which the segment $[z,x]$ makes with the ray from $0$ through $x$, and we define
$$\alpha(x)=\liminf_{z\rightarrow x}\alpha(z, x)\in [0,\pi/2].$$
If $\partial D$  has a tangent hyperplane at $x$ whose normal forms an acute angle $\theta$ with the ray from $0$ through $x$, then $$\alpha(x)=\pi/2-\theta.$$
We say a domain $D$ satisfies the {\it $\alpha$-tangent condition} if for every $x\in\partial D$ we have $\alpha(x)\ge \alpha\in(0,\pi/2]$.
%A domain $D$ satisfies the {\it $\alpha$-tangent condition} if $\partial D$ is smooth a.e. and for the smooth point $x\in\partial D$ we have %$\alpha(x)\ge \alpha\in(0,\pi/2]$.
%In a similar way we define $$\alpha'(x) = \limsup_{z\rightarrow x}\alpha(z, x)\in (0,\pi/2] \ \ \text{ and } \alpha':=\inf_{x\in\partial %
%D}\alpha'(x).$$
\end{subsection}
\begin{subsection}{\bf $\beta$-cone condition \cite{mar2}.} Suppose $D$ is a strictly starlike domain w.r.t. the origin $0$  and let $\beta\in(0,\pi/4]$. We say that $D$ satisfies the {\it $\beta$-cone condition} if the open cone
$$C(x,\beta):=\{z\in\mathbf{C}:|z-x|<|x|,
\left<z-x,x\right>>|x-z||x|\cos \beta\}$$
with vertex $x$ and central angle $\beta$ lies in $D$ whenever $x\in\partial D$. Note that if $D$ satisfies the $\beta$-cone condition, then $D$ is strictly starlike.
\end{subsection}
%===============================================================================
\begin{proposition}\cite{kvw}
A domain $D$ satisfies $\beta$-cone condition if and only if it satisfies $\alpha$-tangent condition.
\end{proposition}
Let $\gamma$ be a smooth starlike Jordan curve
w.r.t.  the origin in $\mathbf C$ such that every tangent line of
$\gamma$ is disjoint from the origin. We will recall some properties
of $\gamma$. Let $s \to r(s)e^{is}$ be the polar parametrization of
$\gamma$. The tangent $t_{s}$ of $\gamma$ at $\zeta=r(s)e^{is}$ is
defined by
$$y=r(s)e^{is}+(r'(s)+ir(s))e^{is} x,\
 \ \ x\in\mathbf{R}.$$ Following the notations in \cite{comp}, the angle
$\alpha_{s}$ between $\zeta$ and the positive oriented tangent at
$\zeta$ is given by
\begin{equation}\label{onesec}
\cos\alpha_{s }=
\frac{r'(s )}{ \sqrt{r^2(s) +{r'}^2(s )}}.
\end{equation}
Hence $$\sin\alpha_{s}=\frac{r(s)}{ \sqrt{r^2(s) +{r'}^2(s
)}}.$$ Consequently
\begin{equation}\label{one}
\cot\alpha_{s}=\frac{r'(s )}{r(s)}.
\end{equation}
Observe that for a smooth starlike Jordan curve $\gamma$ such that
every tangent line of $\gamma$ is disjoint from the origin, we have
\begin{equation*}0<\alpha_1=\min_{t} \alpha_t\le \max_t
\alpha_t =\alpha_2<\pi.\end{equation*} Put
\begin{equation}\label{alal}\alpha_\gamma=\min\{\alpha_1,\pi-\alpha_2\}.\end{equation} Then the domain bounded by $\gamma$  satisfies $\alpha$ ($\alpha_\gamma$)-tangent condition. We will simultaneously say that the curve $\gamma$ satisfies $\alpha-$condition.

Let $G:\mathbf T\to \gamma$ be a continuous locally injective
function from the unit circle $\mathbf T$ onto the star-like Jordan
curve $\gamma$. Then
$$g(t)=\rho(t)e^{i\psi(t)}=G(e^{it}),\, \, t \in [0,2\pi)$$ is a
parametrization of $\gamma$ which represents $g$. If $g$ is an
orientation preserving function then $\psi$ obviously is monotone
increasing. Suppose that $g$ is differentiable. Since
$r(\psi(t))=\rho(t)$, we deduce that $\rho'(t)={r}'(\psi(t))
\cdot {\psi'(t)}$. Hence
\begin{equation}\label{two}
{r'}(\psi(t))=\frac{\rho'(t)}{{\psi'(t)}}.
\end{equation}
By \eqref{two} and \eqref{one} we obtain \begin{equation}
\label{use} \rho'(t) = \rho(t){\psi'} \cot \alpha_{\psi(t) }.
\end{equation}
\begin{theorem}\label{star}
Let $\gamma$ be a smooth starlike Jordan curve with respect to the
origin such that every tangent line of $\gamma$ is disjoint from the
origin. Assume that $G(e^{it})=g(t)=\rho(t)e^{i\psi(t)} :\mathbf
T\to \gamma$ is a homeomorphism.  Let
$$w(z) = |z|G\left(\frac{z}{|z|}\right):\mathbf{C}\to \mathbf{C}.$$ Then the following conditions are equivalent
\begin{itemize}
    \item[(a)] $G$ is bi-Lipschitz,
    \item[(b)] $\psi$ is bi-Lipschitz,
    \item[(c)] $w$ is bi-Lipschitz,
    \item[(d)] $w$ is quasiconformal.
\end{itemize}
Moreover if $\psi$ is $L$ bi-Lipschitz,
then
\begin{itemize}
    \item[(e)]
 $w$ is $(\ell,\mathcal{L})$ bi-Lipschitz continuous, where
$$\mathcal{L} =\frac{|\rho|_\infty}{2\sin\alpha_\gamma}\left(\sqrt{L^2+\sin^2\alpha_\gamma(1-2L)}+\sqrt{L^2+\sin^2\alpha_\gamma(1+2L)}\right),$$\ $$\ell =\frac{4\mathrm{dist}^2(\gamma,0)}{L\mathcal L},$$
and
 \item[(f)] $w$ is $k-$quasiconformal, where $$k =
\sqrt{\frac{L^2+\sin^2\alpha_\gamma(1-2L)}{L^2+\sin^2\alpha_\gamma(1+2L)}}.$$
%$$K =\mathrm{ess\,sup}_{0\le
%t<2\pi}\frac{|{\dot g}+i g|+|{\dot g}-i g|}{|{\dot g}+ig|-|{\dot g}-ig|}\le K_0$$ and
%$$K_0=\frac{\left(\sqrt{L^2 + (1 + 2 L) \sin^2\alpha_\gamma} + \sqrt{
% L^2 + (1 - 2 L) \sin^2\alpha_\gamma}\right)^2}{4L
%\sin^2\alpha_\gamma}. $$
%
\end{itemize}
On the other hand if $w$ is $k-$quasiconformal, then
\begin{itemize}
    \item[(g)]  $\psi$ is $(\frac{1-k}{1+k},\frac{1+k}{1-k})$ bi-Lipschitz
\\
and

\item[(h)] $\sin \alpha_\gamma \ge \frac{1-k^2}{1+k^2}$, where
$\alpha_\gamma$ is defined in \eqref{alal}.
\end{itemize}
\end{theorem}

\begin{remark}
If $\gamma$ is the unit circle, then $\alpha_\gamma=\frac{\pi}{2}$,
and consequently Theorem~\ref{star} contains Theorem~\ref{cir}.
In particular for $\psi$ being an identity map, we obtain the
mapping $w(z)=|z|\rho(t)e^{it}$ which is called \emph{radial
stretching} (see \cite{mar2}). The definition of radial stretching
maps can be applied to several dimensional case as well. In the same
paper Martio and Srebro proved Theorem~\ref{star} for radial
stretching maps in multidimensional setting and for starlike domains
satisfying $\beta-$cone condition.  For a possible variation of the previous
problem to the half-plane we refer to \cite{mate2}.
\end{remark}

\begin{proof} Notice first that $w(z)=\Phi\circ \Psi(z)$, where $\Psi(z)=|z|e^{i\psi(t)}$ and $\Phi(z)=|z|r(t)e^{it}$. By Theorem~\ref{cir} the mapping $\Psi$ is bi-Lipschitz if and only if $\psi$ is bi-Lipschitz with the same bi-Lipschitz constant. Further the radial stretching $\Phi(z)$  is bi-Lipschitz provided that the curve $\gamma$ is smooth and every tangent line of $\gamma$ is disjoint from the
origin (\cite[Theorem~4.8]{kvw}). Thus $G(e^{it})=\Phi\circ \Psi(e^{it})$ is bi-Lipschitz (in the view of the fact that $\Phi$ is a priory bi-Lipschitz) if and only if $\psi$ is bi-Lipschitz. This implies (a) $\Leftrightarrow$ (b) $\Leftrightarrow$ (c). Further we already know that (c) $\Rightarrow$ (d) and it remains to prove the opposite direction. Again by applying \cite[Theorem~4.8]{kvw} we have that $w(z)=\Phi\circ \Psi(z)$ is quasiconformal if and only if $\Psi$ is quasiconformal. By making use of Theorem~\ref{cir} we conclude that this is equivalent to the fact that $\psi$ is bi-Lipschitz.

From \eqref{use}, for $\tau=\psi(t)$ we obtain
$$|{ g'}+i
g|\pm|{ g'}-i
g|={\rho}\left(\sqrt{\frac{(\psi')^2}{\sin^2\alpha_\tau}+1+2{\psi'}}\pm\sqrt{\frac{(\psi')^2}{\sin^2\alpha_\tau}+1-2{\psi'}}\right),$$
and consequently $$(|{ g'}+i g|+|{ g'}-i g|)(|{ g'}+i
g|-|{ g'}-i g|)=4\rho^2{\psi'}.$$ Thus $$|{ g'}+i
g|+|{ g'}-i
g|\le 2\mathcal{L}$$ and $$(|{ g'}+i
g|-|{g'}-i g|)^{-1}\le \frac{\mathcal{L}L}{2\mathrm{dist}^2(\gamma,0)}.$$ By making use of Lemma~\ref{ura},
we obtain $(e)$. Further by \eqref{ww} we have
$$|\mu_w(re^{it})|^2 =
\frac{(1-{\psi'})^2\rho^2+{\rho'}^2(t)}{{(1+{\psi'})^2\rho^2+{(\rho')}^2(t)}}.$$
Since $\dot\rho(t) = \rho(t){\psi'} \cot \alpha_\tau$ it follows
that
\begin{equation}\label{ka}|\mu_w(re^{it})|^2=
\frac{{(\psi')^2}+(1-2{\psi'}){\sin^2\alpha_\tau}}{{(\psi')^2}+(1+2{\psi'}){\sin^2\alpha_\tau}}.\end{equation}
Furthermore, since $L=\max\{|{\psi'}|_\infty,|1/{\psi'}|_\infty\}$, it follows that $L\ge
1$.

Let $$\omega(A)=\frac{A^2+(1-2A)\sin^2\alpha_\tau}{A^2+(1+2A)\sin^2\alpha_\tau}, \ \ \ A\in[1/L,L].$$ Then $$\omega'(A)=\frac{4\sin^2\alpha_\tau(A+\sin\alpha_\tau)(A-\sin\alpha_\tau)}{({A^2+(1+2A)\sin^2\alpha_\tau})^2}.$$ Thus $\omega(A)$ is increasing for $A\in[\sin\alpha_\tau, L]$ and is decreasing for $A\in[1/L, \sin\alpha_\tau]$, provided $1/L<\sin\alpha_\tau$. This implies that $$\omega(A)\le \max\{\omega(1/L),\omega(L)\}.$$
Since  $$\omega(L)-\omega(1/L)=
\frac{\sin^2(2\alpha_\tau)L(L^2-1)}{\left({{L}^2}+(1+2L){\sin^2\alpha_\tau}\right)\left({{1}+(L^2+{2}{L}){\sin^2\alpha_\tau}}\right)}\ge 0,$$ we have that $\max\{\omega(1/L),\omega(L)\}=\omega(L)$ and therefore $$|\mu_w(re^{it})|^2\le \frac{L^2+\sin^2\alpha_\tau(1-2L)}{L^2+\sin^2\alpha_\tau(1+2L)}.$$
Thus $w$ is $k-$quasiconformal where
$$k =
\sqrt{\frac{L^2+\sin^2\alpha_\gamma(1-2L)}{L^2+\sin^2\alpha_\gamma(1+2L)}}.$$
This concludes the proof of (f).

 If $w$ is
$k-$quasiconformal, then $w$ is differentiable almost everywhere. So we can apply \eqref{ka} together with  $|\mu_w(re^{it})|^2\le k^2$, which is equivalent to the following inequalities
$$\frac{1-k}{1+k}\le \psi'\le \frac{1+k}{1-k}\ \ \ \text{and} \ \ \ \frac{\psi'^2(1-k^2)}{2\psi'(1+k^2)-1+k^2}\le \sin^2\alpha_\tau.$$
Since the minimal value of the expression $\frac{A^2(1-k^2)}{2A(1+k^2)-1+k^2}$ for $A\in(\frac{1-k}{1+k},  \frac{1+k}{1-k})$ is attained for $A=\frac{1-k^2}{1+k^2}$ we obtain that $$\sin^2\alpha_\tau\ge \frac{(1-k^2)^2}{(1+k^2)^2},$$ i.e. $$\sin\alpha_\tau\ge \frac{1-k^2}{1+k^2}.$$
This finishes the proof of $(g)$ and $(h)$.
\end{proof}
\begin{remark}
The question arises, which homeomorphism of the unit circle onto a
smooth starlike Jordan curve $\gamma$ induces a radial
quasiconformal mapping with the smallest constant of
quasiconformality. It follows from Theorem~\ref{star}, h) that if there
exists a $K$-quasiconformal radial mapping between the unit disk and
a smooth starlike domain, then $K\ge \cot (\alpha_\gamma/2)$. We expect that, in the notation of Theorem~\ref{star}, the optimal $\psi$ is $\psi(t)\equiv t$.  If instead of $w(z)=|z|F(e^{it})$, we take $w(z)=|z|^\alpha F(e^{it})$, then the constant $K= \cot (\alpha_\gamma/2)$ is attained by the function $w(z)=|z|^{\csc \alpha}F(e^{it})$ (\cite{kvw,{gv}}). To
motivate the previous question, recall the Teichm\"uller problem.
For a given $M$-quasisymmetric selfmapping of the unit circle or
(equivalently) of the real line, find an extension with minimal
constant of quasiconformality. This problem is related to the
extremal quasiconformal mappings. Concerning this topic we refer to the
papers \cite{blmm,mate,reic}.
\end{remark}

\begin{example} Let $0<b\le a$ and  let $$w(z)=|z| \phi(t)=|z| \varphi(e^{it})=|z|\left(\frac{\cos^2
t}{a^2}+\frac{\sin^2 t}{b^2}\right)^{-1/2} e^{it}$$ be a radial
mapping of the unit disk onto the interior of the ellipse
$$E(a,b)=\left\{(x,y):\frac{x^2}{a^2}+\frac{y^2}{b^2}=1\right\}.$$
Then by making use of Corollary~\ref{pams}, we have
\[\begin{split}\mathrm{Lip}(\varphi)=\mathrm{Lip}(\phi) =\left\{
                                                           \begin{array}{ll}
                                                             \frac{2 (a^2 + b^2)^{3/2}}{3\sqrt 3 a b}, & \hbox{if $a^2\ge 2 b^2$;} \\
                                                             a, & \hbox{if $a^2\le 2 b^2$. }
                                                           \end{array}
                                                         \right.
\end{split}\]  On the other hand by Corollary~\ref{rrr} for $a^2> 2 b^2$ we have
$\mathrm{Lip}(\varphi)<\mathrm{Lip}(w),$ because in contrary $\mathrm{Lip}(\varphi)=\mathrm{Lip}(w)=a$ what is not the case.
Moreover
 $ w$ is a $K$ quasiconformal mapping
where \begin{equation}\label{K}K = \frac{a^4 + 6 a^2 b^2 + b^4 +
 \sqrt{14 a^2b^2 + a^4 + b^4} \abs{a^2 - b^2}}{8 a^2
 b^2}.\end{equation}
To show \eqref{K}, we begin by
%
%$$\sin^2 \alpha_\tau = \frac{(b^2 \cos^2 t + a^2 \sin^2t)^2}{(b^4 \cos^2 t + a^4 \sin^2
%t)}.$$
%
$$\frac{\|\nabla w\|^2}{J_w}=
\frac{7 + 2 c^2 + 7 c^4 - b 8 (-1 + c^4) \cos(2 t) + (-1 + c^2)^2
\cos(4 t)}{
 8 (\cos^2 t +
   c^2 \sin^2 t)^2},$$ where $c = \frac ab$.
The minimum is $2$ and is achieved for $t = 0$, $t=\pi$, $t=\pi/2$
or $t=-\pi/2$ and the maximum for $$t = \pm\text{arccos}(\pm
\frac{c}{\sqrt{1 + c^2}}).$$ The maximum is equal to
$$\frac{1 + 6 c^2 + c^4}{4 c^2}.$$ The larger solution of the
equation $$\left(K+\frac 1K\right) =\frac{1 + 6 c^2 + c^4}{4 c^2}$$
is given by \eqref{K}.
\end{example}

\begin{example}
Let $$f(s)=F(e^{is})=\min\left\{\frac{1}{|\sin s|},\frac{1}{|\cos
s|}\right\}e^{is}.$$ Then $f$ is a polar parametrization of the unit
square $Q=\{(x,y): \max\{|x|,|y|\}= 1\}.$ Moreover, by Theorem~\ref{ara} and Corollary~\ref{rrr}, we
have
$$\mathrm{Lip}(f) = \mathrm{Lip}(F)=2<\mathrm{Lip}(w).$$  To obtain the Lipschitz and quasiconformality constant of $w$ we make use of formulas from Theorem~\ref{star}, e), f). But, since the square is not smooth, we cannot directly use Theorem~\ref{star}. However the same formulas holds as well with $\alpha_\gamma=\frac{\pi}{4}$, (\cite[Theorem~4.8]{kvw}) for the special case $\psi(t)\equiv t$. We obtain $$\mathrm{Lip}(w)=\frac 12 (\sqrt{2} +
\sqrt{10})\approx 2.28825$$ and $$ K =\frac{1}{2} (3 + \sqrt{5})\approx 2.61803.$$
\end{example}

\subsection*{Acknowledgement} I am thankful to the referee for numerous corrections,  comments and suggestions  that have improved this paper. I am thankful to Professor Matti
Vuorinen who drew my attention to consider the quasiconformality of
radial extension.  The idea of the proof of Theorem~\ref{ara} is due
to Professor Leonid Kovalev \cite{leonid1}, and we includes his
proof because it is more elegant than the original proof.

\end{document}